\documentclass{gtpart}     % Basic GT/GTM/AGT style
\usepackage{pinlabel}  %%% the recommended graphics+labelling package
\usepackage{graphicx}  %%% the recommended graphics package
\usepackage[all]{xy}
\usepackage{amscd}
\usepackage{tikz}

\usepackage[toc,page]{appendix}
\usepackage{hyperref}
\usepackage[all]{xy}

\mathchardef\dashmod="2D

\usepackage{datetime}
\usepackage{pgf,tikz}
\usetikzlibrary{arrows}

\usepackage[latin1]{inputenc}
\usepackage[T1]{fontenc}

\newcommand{ \inj}{ \hookrightarrow}
\newcommand{ \surj}{ \twoheadrightarrow}
\newcommand{\colim}{\mathrm{colim}}

\newcommand{\tor}{\mathrm{Tor}}

\newcommand{\lambdaq}{\underline{\lambda}}

\newcommand{\tmf}{\mathrm{tmf}}
\newcommand{\mmf}{\mathrm{mmf}}

\newcommand{\G}{{C_2}}
\newcommand{\tmfq}{\tmf_{\G}}
\newcommand{\tmfr}{\mathrm{mmf}_{\R}}
\newcommand{\tmfc}{\mathrm{mmf}_{\C}}

\newcommand{\ko}{\mathrm{ko}}

\newcommand{\hmf}{H\underline{\F}}
\newcommand{\hmz}{H\underline{\Z}}
\newcommand{\hfr}{H\F^{\R}}
\newcommand{\hfc}{H\F^{\C}}

\newcommand{\chfr}{H\F_{\R}}
\newcommand{\chfc}{H\F_{\C}}

\newcommand{\Aq}{\underline{\A}}
\newcommand{\Ar}{\A^{\R}}
\newcommand{\Eq}{\underline{\E}}
\newcommand{\cEr}{\E_{\R}}
\newcommand{\Ac}{\A^{\C}}
\newcommand{\cAr}{\A_{\R}}
\newcommand{\cAc}{\A_{\C}}

\newcommand{\pxi}{\overline{\xi}}
\newcommand{\pz}{\overline{\zeta}}

\newcommand{\sign}{\sigma}
\newcommand{\nor}{\kappa}

\newcommand{\F}{\mathbb{F}}
\newcommand{\N}{\mathbb{N}}

\renewcommand{\S}{\mathbb{S}}

\newcommand{\E}{\mathcal{E}}
\newcommand{\ab}{\mathcal{A}b}

\renewcommand{\sh}{\mathcal{SH}}

\newcommand{\mot}{\mathrm{Mot}}
\newcommand{\shc}{\sh_{\C}}
\newcommand{\shr}{\sh_{\R}}
\newcommand{\shg}{\sh_{C_2}}
\newcommand{\spec}{\mathrm{Spec}}

\newcommand{\cell}{\mathrm{Cell}}

\newcommand{\re}{\mathrm{Re_B}}
\newcommand{\sing}{\mathrm{Sing}}
\newcommand{\overre}{\overline{\re}}
\newcommand{\oversing}{\overline{\sing}}
\newcommand{\rs}{\sing(\S)}

\newcommand{\gm}{\mathbb{G}_m}

\newcommand{\M}{\mathbb{M}}
\newcommand{\A}{\mathcal{A}}

\renewcommand{\P}{\mathcal{P}}

\renewcommand{\J}{\mathcal{J}}

\newcommand{\sur}{/ \hspace{-0.07cm}/}

\newcommand{\univ}{E{\G}}
\newcommand{\tuniv}{\widetilde{\univ}}

\newcommand{\ootimes}{\tilde{\otimes}}

\newcommand{\comp}{\Psi}

\newcommand{\Acomod}{{_{\A_*}}\mathbf{Comod}}
\newcommand{\Aqcomod}{{_{\Aq_*}}\mathbf{Comod}}

\theoremstyle{definition}
\newtheorem{de}{Definition}[section]
\newtheorem{nota}[de]{Notation}

\theoremstyle{plain}
\newtheorem{thm}[de]{Theorem}
\newtheorem{lemma}[de]{Lemma}
\newtheorem{pro}[de]{Proposition}
\newtheorem{cor}[de]{Corollary}

\newtheorem{warning}[de]{Warning}

\newtheorem*{thm*}{Theorem}
\newtheorem*{lemma*}{Lemma}
\newtheorem*{pro*}{Proposition}
\newtheorem*{cor*}{Corollary}

\theoremstyle{remark}
\newtheorem{rem}[de]{Remark}
\newtheorem{ex}[de]{Example}

\newtheorem{conj}[de]{Conjecture}

\title[$\mmf$]{Motivic modular forms from equivariant stable homotopy theory}

\author{Nicolas Ricka}
\givenname{Nicolas}
\surname{Ricka}
\address{Wayne State University, Department of Mathematics, 656 W. Kirby, Detroit, MI 48202}
\email{nicolas.ricka@wayne.edu}
\urladdr{https://sites.google.com/site/nicoricka/}

\keyword{equivariant stable homotopy theory, motivic stable homotopy theory, topological modular forms, motivic modular forms, Steenrod algebra}
\subject{primary}{msc2010}{55P42, 55S10, 55N91,14F42,55T15}
\subject{secondary}{msc2010}{ }

\arxivreference{}
\arxivpassword{}

\begin{document}

\begin{abstract}
In this paper, we produce a cellular motivic spectrum of motivic modular forms over $\R$ and $\C$, answering positively to a conjecture of Dan Isaksen. This spectrum is constructed to have the appropriate cohomology, as a module over the relevant motivic Steenrod algebra.
We first produce a $\G$-equivariant version of this spectrum, and then use a machinery to construct a motivic spectrum from an equivariant one. We believe that this machinery will be of independent interest.
\end{abstract}

\maketitle

\section*{Introduction}

The $E_{\infty}$-ring spectrum  $\tmf$, constructed by Hopkins and Miller, is of great importance in today's stable homotopy theory. For instance, the Adams spectral sequence computing $\tmf_*$ sees more non-trivial elements from the homotopy of the sphere than the Adams spectral sequence for $\ko_*$. Indeed, if $\A$ denotes the modulo $2$ Steenrod algebra, and $\A(n)$ denotes its subalgebra generated by the first $(n+1)$ generators $Sq^{2^i}$ for $0 \geq i \geq n$, then $$H\F^*(\ko) \cong \A\sur\A(1)$$ whereas 
\begin{equation} \label{eq:cohtmf}  H\F^*(\tmf) \cong \A\sur\A(2). \end{equation} As a consequence of the Hopf invariant one question, there cannot possibly exist a spectrum $X$ whose cohomology is $\A\sur\A(n)$ for any $n \geq 3$. Thus, $\tmf$ plays a particular role as it is the last one possible. Its particular cohomolohy, together with the Adams spectral sequence implies that the spectrum $\tmf$ has a large Hurewicz image (for instance, larger than the Hurewicz image of $\ko$. The interested reader who needs motivations from a different perspective can read \cite{TMFbook}.

Another approach to understand the Adams spectral sequence is to compare it to its analogues in other stable categories than the classical stable homotopy category.
For example, comparing the classical Adams spectral sequence to its motivic analogue, in the $\mathbb{A}^1$-stable category of motives, Isaksen \cite{I14} is able to do some new computations of the classical stable stems.

Motivated by these two phenomena, we are interested in this paper by the existence of a motivic version of the spectrum of topological modular forms $\tmf$ in the $\mathbb{A}^1$-stable homotopy category over $\spec(\R)$ and $\spec(\C)$. The existence of such a spectrum was conjectured in \cite{Isa09}.

Let $\cAr$ be the motivic Steenrod algebra, and $\cAr(2)$ be its subalgebra generated by the motivic Steenrod squares $Sq^1, Sq^2, Sq^4$ (see \cite{Voe03b}). The main result of this paper is the construction of a motivic  spectrum $\tmfr$ over $\spec(\R)$ whose motivic cohomology is  $\cAr\sur\cAr(2)$. This also gives a model for motivic modular forms over $\spec(\C)$ by pullback (see subsection \ref{sub:tmfr}). \\

The main result of the paper is the following.

\begin{thm*}[Theorem \ref{thm:tmfr}, Corollary \ref{cor:cohommmfc}]
There exist motivic spectra $\tmfc$ and $\tmfc$ in the stable motivic category over $\spec(\C)$ and $\spec(\R)$ respectively, whose cohomology is
$$\chfc^{*}(\tmfc) \cong \cAc \sur \cAc(2)$$
and
$$\chfr^{*}(\tmfr) \cong \cAr \sur \cAr(2),$$
respectively.
\end{thm*}

The construction is indirect, and  only relies on the existence of $\tmf$ as a ring spectrum, satisfying equation \eqref{eq:cohtmf}.
We decompose the construction into two main steps: first, we build a $\G$-equivariant version of $\tmf$ which we call  $\tmfq$, where $\G$ denotes the group with two elements. This spectrum $\tmfq$ is constructed from its Tate diagram. The background on generalized Tate cohomology we need is recalled in \ref{sub:tatediag}. The determination of the Tate spectrum of $\tmf$ is the main result of \cite{BR17}.

The second step is to build the motivic modular forms spectrum over $\spec(\R)$ from $\tmfq$. Let $\shr$ denote the $\mathbb{A}^1$-stable homotopy category over $\spec(\R)$. In \cite{HO14}, the authors consider an equivariant version of the Betti realization functors (the non-equivariant version was introduced by Morel and Voevodsky in \cite{MV01,Voe02}). This functor has a right adjoint denoted $\sing$, and a section $c^*$ (again, see \cite{HO14}, or the recollections in subsection \ref{sub:comparison}). For a motivic spectrum $X \in \shr$, we denote by $X^{\wedge}_{\sing(\S)}$ the $\sing(\S)$-nilpotent completion of $X$, where $\S$ is the sphere spectrum.

We show that, for any spectrum $E \in \shg$ with a nice enough cohomology (see the hypothesis of Theorem \ref{thm:hmote}), the motivic cohomology of $(c^*E)^{\wedge}_{\sing(\S)}$ is obtained from the $\G$-equivariant cohomology of $E$ via an extension of scalars.
As a consequence, we define $\tmfr:=(c^*\tmfq)^{\wedge}_{\sing(\S)}$.

Let $\sh$ be the classical stable homotopy category, $\shg$ be the $\G$-equivariant stable homotopy category and $\shc$ (resp. $\shr$) be the cellular $\mathbb{A}^1$-stable homotopy category over $\spec(\C)$ (resp. $\spec(\R)$). The localization we consider to obtain the cellular category is described in subsection \ref{sub:cellularization}. The situation is summarized in the following diagram.

\begin{equation} \label{eq:square}
\xymatrix{ \tmfr \quad \in \ar@{~>}@/_1pc/[d]_{pullback}& \shr \ar[r]^{\re} \ar[d]^{p^*} & \shg \ar[d]^{i^*} & \ni \quad \tmfq \ar@{~>}@/_2pc/[lll]_{(c^*(-))^{\wedge}_{\sing(\S)}} \\
\tmfc \quad \in & \shc \ar[r]^{\re} & \sh &  \ni \quad \tmf \ar@{~>}@/_1pc/[u]_{Tate} }
\end{equation}
where $\re : \shc \rightarrow \sh$ and $\re : \shr \rightarrow \shg$ are the Betti realization functors. \\

The first part is devoted to the construction of motivic spectra from equivariant ones. Our construction of $\tmfr$ from $\tmfq$ will follows from these considerations, but the work done in this section is intended to be more general than this, and would apply to any spectrum whose homology is free over the homology of a point, with a smallness condition on the generators (see the hypothesis of Theorem \ref{thm:hmote}). In Section \ref{sec:prelim}, we recall the material we need about the equivariant and motivic stable homotopy categories, and fix some notations. In Section \ref{sec:descent}, we show that there is a way to describe the $\G$-equivariant stable homotopy in a way that is completely internal to the category $\shr$. This is the subject of the following theorem. Recall that $\re$ denotes the Betti realization functor. We denote also by $\sing$ its right adjoint.

\begin{thm*}[{Theorem \ref{thm:qshinmot}}]
The adjunction
$$ \re : \shr \leftrightarrows \shg : \sing$$
factors as

$$ \xymatrix{ \shr \ar@<.05cm>[rr]^{\re} \ar@<-.05cm>[dr]_{\rs \wedge (-)} & & \shg \ar@<.05cm>[ll]^{\sing} \ar@<.05cm>[dl]^{\oversing}\\
& \rs \dashmod mod(\shr) \ar@<.05cm>[ru]^{\overre} \ar@<-.05cm>[lu]_{U} & }$$

where the adjunct pair
$$ \overre : \rs\dashmod mod(\shr) \leftrightarrows : \shg :\oversing$$
is a Quillen equivalence .
\end{thm*}

Note that the corresponding non-equivariant statement is a direct consequence of the computations of Dan Isaksen \cite{I14} (see Remark \ref{rk:classicalrewriting} about this).

In particular, this exhibits the $\G$-equivariant stable homotopy category as a category of modules over some ring spectrum $\sing(\S) \in \shr$. Thus, given a $\G$-equivariant spectrum $E$, we can ask whether it is induced from a motivic spectrum, \textit{i.e.} if it is in the image of the functor $\sing(\S) \wedge(-)$.
We build in Definition \ref{de:mot} a functor that assigns to a $\G$-equivariant spectrum $E$ its closest motivic spectrum over $\spec(\R)$.  

We then turn to the study of the motivic homology of the spectra $(c^*(E))^{\wedge}_{\sing(\S)}$ in terms of equivariant homology. Denote by $\hmf_*$ the equivariant homology functor, and $\hfr_*$ the motivic homology functor. On both sides, there is a dual Steenrod algebra of cooperations, and homology is a functor that takes values in the category of comodules over this coalgebra (in both settings). The determination of the $\G$-equivariant version of the Steenrod algebra is due to Hu-Kriz \cite{HK01}, and the motivic version of the Steenrod algebra has been determined by Voevodsky \cite{Voe03b}.

\begin{thm*}[{Theorem \ref{thm:cohmote} and Theorem \ref{thm:hmote}}]
Let $E \in \shg$ be a $\G$-spectrum whose equivariant homology is free as a module over the homology of a point. Suppose moreover that the smallness condition of Theorem \ref{thm:cohmote} holds. Then, there is a natural isomorphism of comodules over the dual motivic Steenrod algebra
$$\hfr_{*}(\mot(E)) \cong \hfr_{*} \otimes_{\hmf_{*}} \hmf_{*}(E).$$
\end{thm*}

This reduces the construction of motivic modular forms over $\spec(\R)$ to the construction of the $\G$-equivariant version of topological modular forms. \\

In the second part of the paper, we build $\tmfq$. This is a more technical part and relies heavealy on the machinery developped in \cite{Mah81,DM84,DM86} as well as explicit computations. The key step is to compute the Tate spectrum of the classical spectrum $\tmf$.

\begin{thm*}[{\cite[Theorem 1.1]{BR17}}]
There is a weak equivalence of spectra
$$t(\tmf) \cong \prod_{i \in \Z} \Sigma^{8i} \ko =: \ko((x^8))$$
where $x^8$ is in degree $8$.
\end{thm*}
The knowledge of the Tate spectrum of $\tmf$ enables us to show

\begin{thm*}[{Theorem \ref{thm:cohtmfq}}]
There is an isomorphism of $\Aq$-modules
$$\hmf^{*}(\tmfq) \cong \Aq\sur\Aq(2).$$ 
\end{thm*}

Then, applying Theorem \ref{thm:cohmote} and Theorem \ref{thm:hmote} to Theorem \ref{thm:cohtmfq} gives the spectrum of motivic modular forms over $\spec(\R)$.

We conclude by the construction of the spectrum of motivic modular forms over $\spec(\C)$  by pullback from $\tmfr$ in subsection \ref{sub:tmfc}

An interesting phenomenon appears here: although the motivic stable category over $\spec(\C)$ is simpler, the technique presented here does not produce directly the spectrum $\tmfc$, and a detour by the equivariant stable homotopy category is necessary. The philosophical reason why this happens is that the motivic Steenrod algebra over $\spec(\R)$ and the $\G$-equivariant Steenrod algebra are essentially the same (modulo an extension of scalars), whereas the classical and motivic over $\spec(\C)$ are very different from one another (see \cite{Voe03b}).

\textbf{Acknowledgments} The author thanks Mike Hill for suggesting the construction of $\tmfq$ by its Tate square, and Dan Isaksen for numerous discussions about the motivic part of this project and closely related matters. The author thanks  Tom Bachmann for spotting a loophole in the proof of Lemma 2.1 in an earlier  version.

\tableofcontents

\part{A refined comparison between the motivic and equivariant stable homotopy categories} \label{part:comparison}

\section{Preliminaries} \label{sec:prelim}

\subsection{Set-up} \label{sub:setup}
As mentioned in the introduction, our main concern in this part is the relationship between classical and motivic homotopy theories. Precisely, we want to refine the known relationship between the stable motivic homotopy category over $\spec(\R)$ and and the $\G$-equivariant stable homotopy category provided by the Betti realization. 

Both are tensor triangulated categories, which arise as homotopy categories of stable monoidal model categories, the tensor product being the smash product $\wedge$. We denote the unit of both these monoidal categories by $\S$. Moreover, the monoidal structure is closed, and we denote by $F(-,-)$ the morphism object in each one of these categories.

The category of $\G$-spectra is compactly generated by the representation spheres: these are the one point compactification of finite virtual orthogonal representations of $\G$. In particular, any object can be approximated up to weak equivalence by a cellular object build from such spheres.

Let's denote by $S^{V}$ the one point compactification of the orthogonal representation $V$. By elementary representation theory, every such representation is of the form $n + m \sign$, that is the direct sum of $n$ times the trivial representation and $m$ times the sign representation. Let $V,W$ be virtual real orthogonal representations of $\G$. There is an isomorphism 
$S^{V \oplus W} \cong S^V \wedge S^W$, which is natural in $W,V$. This gives a monoidal functor
\begin{equation} \label{eq:rogqsh}
RO(\G) \rightarrow \shg,
\end{equation}
where $RO(\G)$ is the Grothendieck group of finite orthogonal representations of $\G$, under the direct sum.

\begin{de} \label{de:salpha}
We denote by $S_{\G}^{(1,0)}$ and $S_{\G}^{(1,1)}$, or simply $S^{(1,0)}$ and $S^{(1,1)}$, if there is no possible ambiguity, the images of the one dimensional trivial and sign representations respectively.
\end{de}

The suspension functor $\Sigma : \shg \rightarrow \shg$, which is part of the triangulated structure on $\shg$ coincides with the functor $S^{(1,0)} \wedge (-)$. The sphere $S^{(1,1)}$ is of purely equivariant nature.

\textbf{Warning:} note that our grading convention is similar to the one usually adopted among the $\mathbb{A}^1$-stable homotopy theorists. However, this differs a little from the conventions of the foundational article \cite{HK01} in the $\G$-equivariant setting.

The category $\shr$ also contains two flavors of one dimensional spheres, $S^{(1,0)}$, which is the suspension of the unit $\S \in \shr$, and another one, $S^{(1,1)} = \gm$, the multiplicative group. Again, this provides a monoidal functor
\begin{equation} \label{eq:zzmotr}
\Z \oplus \Z \rightarrow \shr,
\end{equation}
which sends $(n,m)$ to $S^{(n,m)}$.

\begin{rk}{Remark}
The spheres $S^{(n,m)} \in \shr$ does not generate the whole category of motivic spectra over $\spec(\R)$. Later on we will work in the cellular category to avoid complications coming from this. The cellular category being a right Bousfield localization of the motivic category, it contains less information than the latter. However, it is clear from \cite{I14} that this category is still suitable for explicit computations, in particular when it comes to investigate the motivic Adams spectral sequence.
\end{rk}

\subsection{Comparison functors} \label{sub:comparison}
The comparison between the motivic and equivariant settings is classically done using the Betti realization functor
\begin{equation}
\re : \shr \rightarrow \shg,
\end{equation} 
the informations we need about this functor is in \cite{HO14}, although Betti realization has been set up by Morel and Voevodsky \cite[p.122]{MV01}.
We will reformulate the $\G$-equivariant Betti realization functor $\re : \shr \rightarrow \shg$ and its right adjoint $\sing$ in a more convenient way in \ref{thm:qshinmot}.

Let us first recall the definitions of these functors, and the basic properties that we need in this paper. The material contained in this subsection is taken directly from \cite{HO14}.

\begin{de}[{\cite[Section 4.4]{HO14}}]\label{de:sing}
Let $\sing : \shg \rightarrow \shr$ be the stable functor induced by $$\sing(E)(X) = \hom_{\G}(X(\C),E),$$ where $X(\C)$ is the complex points of the motivic space $X$ over $\spec(\R)$, together with its involution coming from the complex conjugation.
\end{de}

\begin{pro}[{\cite{HO14}}] \label{pro:resing}
There is a Quillen adjunction
$$ \re : \shr \leftrightarrows \shg : \sing.$$
Moreover, the Betti realization functor $\re$ is strong symmetric monoidal, and takes the following values on spheres:\begin{itemize}
\item $\re(S^{(1,0)}) = S_{\G}^{(1,0)}$,
\item $\re(S^{(1,1)}) = S_{\G}^{(1,1)}$.
\end{itemize}
\end{pro}

\begin{proof}
The fact that $\re$ is strong symmetric monoidal is \cite[Proposition 4.7]{HO14}.
By \cite{HO14}, at the beginning of Section 4.4, $\re(S^{(1,0)}) = S^{(1,0)}$ and $\re(S^{(1,1)}) = S^{(1,1)}$.
\end{proof}

In particular, the behaviour of $\re$ is not so mysterious on objects which are built from the motivic spheres (say, finite cellular objects with respect to $S^{(n,m)}$): it sends motivic spheres to motivic spheres (see \cite{HO14}) and pushouts to pushouts (as it is a left Quillen adjoint).

\begin{de}[{\cite[Section 2]{HO14}}]
Let $c^* : \shg \rightarrow \shr$ be the constant simplicial presheaf functor.
\end{de}

\begin{pro}[{\cite{HO14}}] \label{pro:secresing}
The functor $c^*$ has a right adjoint
$$c^* : \shg \leftrightarrows \shr : c_*,$$
and these satisfy
$$\re c^* \cong id_{\shg} \cong c_* \sing.$$
\end{pro}

\begin{proof}
The existence of the pair $(c^*,c_*)$ is formal (see \cite{HO14}). The formula $\re c^* \cong id_{\shg}$ is taken from \cite{HO14} and the last natural weak equivalence follows by uniqueness of adjoints.
\end{proof}

\subsection{Homotopy, homology, Steenrod algebras} \label{sub:steenrod}

Since the equivariant spheres are in the image of the strongly monoidal functor \eqref{eq:rogqsh}, they belong to the Picard group of $\shg$. Therefore, homotopy groups are naturally graded over the representation ring (\textit{i.e.} $RO(\G)$-graded) .
The exact same discussion can be repeated replacing \eqref{eq:rogqsh} by  \eqref{eq:zzmotr}.

\begin{de}
We use the following notations for the stable homotopy classes of maps in the various categories into play:\begin{enumerate}
\item for $E,F \in \sh$, and $n \in \N$, denote by $[E,F]_{-n}$, or $[E,F]^{n}$ the abelian group of stable homotopy classes of maps $E \rightarrow \Sigma^n F$,
\item for $E,F \in \shg$, and $(n-m) + m \sign \in RO(\G)$, denote by $[E,F]^{\G}_{(-n,-m)}$, or $[E,F]_{\G}^{(n,m)}$ the abelian group of stable homotopy classes of maps $E \rightarrow \S^{(n,m)} \wedge F$,
\item for $X,Y \in \shr$, $[X,Y]^{\R}_{(-n,-m)} = [X,Y]_{\R}^{(n,m)}$ is the stable homotopy classes of maps $X \rightarrow S^{(n,m)} \wedge Y$.
\end{enumerate}
Homotopy groups are denoted $\pi_*$, $\pi_*^{\G}$, and $\pi_*^{\R}$ with evident notations, and represent the functor $[S,-]_*$, with values in appropriately graded abelian groups. 
\end{de}

By adjunction, there is a natural isomorphism
\begin{equation} \label{eq:pistaradj}
\pi^{\G}_*(X) \cong \pi_*^{\R}(\sing(X)),
\end{equation}
for all $X \in \shg$.

The particular bigrading we have chosen for homotopy groups induces a bigrading on homology and cohomology groups, since for any $\G$-spectrum $E$ (resp. $E \in \shr$), $E$-cohomology is the functor $[-,E]_{\G}^* : \shg^{op} \rightarrow \ab^{RO(\G)}$, where the target is the category of $RO(\G)$-graded abelian groups (resp. $[-,E]^* : \shg^{op} \rightarrow \ab^{\Z^2}$). The same remark applies to homology functors.

In all of the categories in play there is a particular spectrum which is at the center of this paper: ordinary cohomology with coefficients in $\F$. These spectra are crucial when investigating the stable stems, since there is a "computable" Adams spectral sequence associated to each one of them, converging to $\pi_*(\S)$ in any of the category $\sh$, $\shg$, and $\shr$ (see \cite{HK01}, \cite{I14} for the two least classical ones).
 
In the stable homotopy category, the Postnikov coconnective cover of the generator $\S$ is $H\Z$, and killing $2$ gives the desired spectrum $H\F$ (see \textit{loc cit}).

In the $\G$-equivariant stable homotopy category, an analogous construction gives the Eilenberg-MacLane spectrum we are interested in: there is a natural generalization of the Postnikov tower, called the slice tower (see \cite{HHR}). The coconnective cover of the sphere spectrum $\S$ is $\hmz$, and killing the multiplication by $2$ on this spectrum produces the desired $\hmf$.

The motivic analogue $\hfr$ can be similarly described, using the motivic slice filtration instead of the $\G$-equivariant one (this is the main result of \cite{Voe03b}).

Finally, the motivic and equivariant versions of the Eilenberg-MacLane spectra are related to each other via Betti realization.

\begin{pro}[{\cite{HO14}}] \label{pro:rehfr}
There is a weak equivalence of ring spectra
$$ \hmf \cong \re\hfr.$$
\end{pro}

Once the Eilenberg-MacLane spectrum is constructed, it is natural to investigate the corresponding Steenrod algebra. It turns out that, since the cohomology of $S^{(0,0)}$ is not a field in the motivic and equivariant setting, the dual object is more structured (although it is merely a flat Hopf algebroid, and not a Hopf algebra as in the classical case).

\begin{pro}[Ullman, Hoyois]
The spectra $H\F \in \sh$, $\hmf \in \shg$, and $\hfr \in \shr$ are commutative ring spectra.
\end{pro}

\begin{proof}
for $H\F$, this is classical. The equivariant case is a consequence of \cite[Theorem 1.3]{Ul13}, and $\hfr$ being a commutative ring spectrum is \cite[Paragraph 4.2]{Hoy12}).
\end{proof}

To a commutative ring spectrum, there is a natural way to associate a commutative Hopf algebroid of cooperations (see \cite[Proposition 2.2.3]{Ra86}).

It turns out that, for any of the spectra $H\F$, $\hmf$, and $\hfr$, the associated Hopf algebroid of cooperations in homology is flat. These are denoted respectively: \begin{enumerate}
\item $(\F, \A)$,
\item $(\hmf_*,\Aq)$,
\item and $(\hfr_*,\Ar)$.
\end{enumerate}

There is a well-known isomorphism
\begin{equation}
\A \cong \F[\pxi_i, i \geq 1]
\end{equation}
as a commutative algebra, where $|\pxi_i| = 2^i -1$.

The diagonal of this Hopf algebroid is given by 
\begin{equation}
\Delta(\pxi_i) = \sum_{k = 0}^i \pxi_k^{2^{i-k}} \otimes \pxi_k.
\end{equation}

We use the notation $\pxi$ instead of the more classical one $\xi$ to emphasize the difference between the non-equivariant and equivariant Steenrod algebras.

Hu and Kriz has identified the corresponding $\G$-equivariant object in \cite{HK01}, and Voevodsky gives the motivic analogue in \cite{Voe03b}. We recall here the structure of these two objects, starting by the coefficient rings of equivariant and motivic homology theory.

\begin{pro}[Hu-Kriz, Voevodsky] \label{pro:coeffring}
The coefficient rings of motivic and $\G$-equivariant cohomology theories are: \begin{enumerate}
\item $\hfr_{*} = \F[\rho, \tau]$, with the evident ring structure,
\item $\hmf_{*} = \F[\rho, \tau] \oplus \nor \F[\rho^{-1},\tau^{-1}]$. The ring structure of the latter being the square zero extension of $\F[\rho, \tau]$ by the $\F[\rho, \tau]$-module $\nor \F[\rho^{-1},\tau^{-1}]$.
\end{enumerate}
The grading is $|\rho| = (-1,-1)$, $|\tau| = (0,-1)$, $|\nor| = (0,2)$.
\end{pro}

\begin{nota} \label{nota:manddm}
Let $\M$ be the $\F$-algebra $\hfr_{*} = \F[\rho, \tau]$, and denote $D\M$ the $\F$-linear graded dual of $\M$.
In particular $\hmf_{*} = \M \oplus \nor D\M$.
\end{nota}

\begin{pro}[{\cite[Theorem 6.41]{HK01},\cite[Theorem 12.6 and Lemma 12.11]{Voe03}}] \label{pro:steenrod}
The $\G$-equivariant dual Steenrod algebra is the commutative Hopf algebroid $(\hmf_*,\Aq_*)$, where
\begin{equation*}
\Aq_* \cong \hmf_*[\tau_i, \xi_{i+1}, i \geq 0]/(\tau_i^2 + \rho\tau_i +  \eta_R(\tau)\xi_{i+1})
\end{equation*}
as a $\hmf_*$-algebra, and $\eta_R(\tau) = \rho \tau_0 + \tau$. Moreover, the diagonal are given by the formul\ae
\begin{equation*}
\Delta(\tau_i) = \tau_i \otimes 1 + \sum_{k=0}^i \ \xi_{i-k}^{2^k} \otimes \tau_{k},
\end{equation*}
and
\begin{equation*}
\Delta(\xi_i) = \sum_{k=0}^i \ \xi_{i-k}^{2^k} \otimes \xi_{k}.
\end{equation*}
the degree are $|\xi_i| = (2(2^i-1),2^i-1)$ and $|\tau_i|=(2(2^i-1)+1,2^i-1)$.
\vspace{.5cm}
The motivic dual Steenrod algebra is
$(\M,\Ar)$, where
\begin{equation*}
\Ar_* \cong \M[\tau_i, \xi_{i+1}, i \geq 0]/(\tau_i^2 + \rho\tau_i +  \eta_R(\tau)\xi_{i+1})
\end{equation*}
Moreover, the diagonal are given by the formul\ae
\begin{equation*}
\Delta(\tau_i) = \tau_i \otimes 1 + \sum_{k=0}^i \ \xi_{i-k}^{2^k} \otimes \tau_{k},
\end{equation*}
and
\begin{equation*}
\Delta(\xi_i) = \sum_{k=0}^i \ \xi_{i-k}^{2^k} \otimes \xi_{k}.
\end{equation*}
The degrees are the same as in the $C_2$-equivariant case.
\end{pro}

As we are interested in cohomology computations, a crucial property for us is the relationship between motivic cohomology and equivariant cohomology. The following result will be the starting point of the comparison.

\begin{rem} \label{rk:steenrodalgebras}

The Hopf algebroid $\Aq$ can be expressed as an extension of $\Ar$:
$$\Aq = \hmf_* \otimes_{\hfr_*} \Ar.$$
This simple observation, together with Theorem \ref{thm:hmote} is a reason why it is easier to build a spectrum in $\shr$ with prescribed cohomology once we know how to do it in $\shg$.

Note that the corresponding statements relating the dual Steenrod algebras for $H\F$ and $\hfc$ are utterly false.
\end{rem}

\subsection{Cellularization} \label{sub:cellularization}

\begin{de}\label{de:cellular}
Let $\shr \rightarrow \cell$ be the right Bousfield localization of $\shr$ where the weak equivalences are maps which induce an equivalence in bigraded homotopy groups.
\end{de}

Note that this right Bousfield localization exists since $\shr$ is right proper and combinatorial. %this is Lemma 7.8 in the book local homotopy theory.

\begin{rem} \label{rk:cellularrepl}
Of course, by definition of the cellular category, functors as (bigraded) homotopy or (bigraded) (co)homology do not see the difference between an object and its cellular replacement. In particular, this category is well-suited for computational purposes (this is the category in which \cite{I14} takes place for example).
\end{rem}

\begin{pro}
The model category $\cell$ satisfies the following properties: \begin{itemize}
\item it is a stable monoidal closed model category,
\item the functor $\shr \rightarrow \cell$ is strongly monoidal.
\end{itemize}
\end{pro}

\begin{proof}
This is the content of \cite{BR12b} for $K = \{ S^{(n,m)}, n,m\in\Z \}$:\begin{enumerate}
\item \cite[Theorem 4.1]{BR12b} for the existence (this is originally a result of Hirchhorn),
\item \cite[Proposition 4.6]{BR12b}  for stability, since $K$ is obviously stable in our case, in the sense of \cite[Definition 4.1]{BR12b},
\item \cite[Theorem 6.2]{BR12b} for the assertion regarding the monoidal structure, since $K$ is trivially monoidal in our case, in the sense of \cite[Definition 6.1]{BR12b}. \end{enumerate}
\end{proof}

\begin{warning} \label{war:cellular}
By remark \ref{rk:cellularrepl} and since we are interested in computations in homotopy and homology of cellular objects, we now restrict ourselves to the cellular category. Note that in particular, \begin{itemize}
\item $\shr$ and $\shc$ denotes the appropriate cellular categories,
\item $\re$ and $\sing$ denotes the factorizations of these functors through the cellular category namely, compose $\sing$ with the cellularization, and observe that $\re$ has to factor through $\cell$ since it has a left section (see Proposition \ref{pro:secresing}). 
\end{itemize}
\end{warning}

We have now all the tools we need to start our identification of $\shg$.

\section{Descent and cohomology} \label{sec:descent}

\subsection{An identification of $\shg$}  \label{sub:idqsh}

The objective of this section is to prove Theorem \ref{thm:qshinmot}. Roughly speaking, we want a Quillen equivalence $\shg \cong \rs\dashmod mod(\shr)$ such that the pair $(\re,\sing)$ is identified through this equivalence with $(\rs \wedge (-), U)$: respectively the extension of scalars and the forgetful functor.

Certainly, if this is the case, the following are satisfied
\begin{enumerate}
\item $\sing$ commutes with homotopy colimits,
\item $\sing$ is a conservative functor (\textit{i.e.} sends weak equivalences to weak equivalences),
\item $\re$ and $\sing$ satisfies the projection formula, that is for any $E \in \shg$ and $X\in \shr$, the natural map \begin{equation*} \sing(E) \wedge X \rightarrow \sing(E \wedge \re(X)) \end{equation*} is a weak equivalence.
\end{enumerate}

We will show these properties in the next few lemmas before proving Theorem \ref{thm:qshinmot}.

\begin{lemma} \label{lemma:holim}
The functor $\sing$ commutes with homotopy colimits.
\end{lemma}

\begin{proof}
In $\shg$, every colimit is built from binary sums, cofibers, and filtered colimits. Since $\shg$ is a stable category, binary sums and products coincide, and cofibers and fibers coincide up to a shift. Since the functor $\sing$ is a right adjoint, it commutes with these constructions. It remains to show that $\sing$ commutes with filtered colimits.

Let $X = \colim_{i \in I} X_i$. There is a canonical map $\phi : \colim_{i \in I} \sing(X_i) \rightarrow \sing(X)$. We show that this is a weak equivalence in the cellular category. To this end, we need to check that it induces an isomorphism in bigraded homotopy groups. Let $(n,m)$ be a pair of integers. Consider the following commutative diagram, in which the vertical arrows are isomorphisms:
$$\xymatrix{ \pi^{\R}_{(n,m)}(\underset{i \in I}{\colim}\  \sing(X_i)) \ar[r]^{\phi} \ar[d]^{\cong}&  \pi^{\R}_{(n,m)}(\sing(X)) \ar[d]^{\cong} \\
\underset{i \in I}{\colim}\  \pi^{\R}_{(n,m)}(\sing(X_i)) \ar[d]^{\cong} & \pi^{\G}_{(n,m)}(\underset{i \in I}{\colim}\  X_i) \ar[d]^{\cong} \\
\underset{i \in I}{\colim}\  \pi^{\G}_{(n,m)}(X_i) \ar[r]^{=}& \underset{i \in I}{\colim}\  \pi^{\G}_{(n,m)}(X_i),}$$
the vertical isomorphisms comes from the isomorphism given in equation \eqref{eq:pistaradj}, and compactness of $\S^{(n,m)}$ in both the $\G$-equivariant and the motivic setting.
\end{proof}

\begin{lemma} \label{lemma:conservative}
The functor $\sing$ is conservative.
\end{lemma}

\begin{proof}
Let $f: E \rightarrow F$ be a morphism of $\G$-spectra. Suppose that $\sing(f)$ is a weak equivalence. By Proposition \ref{pro:secresing}, $f = c_* \sing (f)$ is then a weak equivalence.
\end{proof}

\begin{lemma} \label{lemma:projection}
Let $E\in \shg$ and $X\in\shr$. The natural map 
 $$\sing(E) \wedge X \rightarrow \sing(E \wedge \re(X))$$
 is a weak equivalence.
\end{lemma}

\begin{proof}
Since we are working in the cellular category, it suffices to check the desired weak equivalence for the generators, which are $S^{(1,0)}$ and $S^{(1,1)}  \in \shg$, and $S^{(1,0)}$ and $S^{(1,1)} \in \shr$. The result is trivial in this case because of the values of $\re$ on motivic spheres.

One concludes using that both sides commutes with homotopy colimits in both variables.
\end{proof}

\begin{thm} \label{thm:qshinmot}
The adjunction
$$ \re : \shr \leftrightarrows \shg : \sing$$
factors as

$$ \xymatrix{ \shr \ar@<.05cm>[rr]^{\re} \ar@<-.05cm>[dr]_{\rs \wedge (-)} & & \shg \ar@<.05cm>[ll]^{\sing} \ar@<.05cm>[dl]^{\oversing}\\
& \rs \dashmod mod(\shr) \ar@<.05cm>[ru]^{\overre} \ar@<-.05cm>[lu]_{U} & }$$

where the adjunct pair
$$ \overre : \rs\dashmod mod(\shr) \leftrightarrows : \shg :\oversing$$
is a Quillen equivalence .
\end{thm}

\begin{proof}
We argue using \cite[Proposition 5.29]{MNN15}. The hypothesis of this proposition are Lemmas \ref{lemma:holim}, \ref{lemma:conservative}, and \ref{lemma:projection}.
\end{proof}

\begin{rem}
Theorem \ref{thm:qshinmot} is not so surprising, and there is an analogous equivalence of homotopy theories between modules over $S[\tau^{-1}]$ in $\shc$ and $\sh$ (using a similar argument to the one presented here in the real case, this is a direct consequence of Dan Isaksen's computation, see \cite{Isa09}).

However, remark \ref{rk:steenrodalgebras} is the philosophical reason why when it comes to studying the action of the Steenrod algebras on both sides of this equivalence, the comparison between $\shr$ and $\shg$ is far closer than the one between $\shc$ and $\sh$.
\end{rem}

\begin{rem}
The following alternative proof of Theorem \ref{thm:qshinmot} has been suggested by Tom Bachmann. We include the outline here for the curious reader.
Betti realization and $\sing$ induces an adjunction $$ \overre : \rs\dashmod mod(\shr) \leftrightarrows : \shg :\oversing$$
for formal reasons (see \cite[5.24]{MNN15}). 
Now, by \cite[Example 5.25]{MNN15}, the hypothesis of \cite[Lemma 21]{Bac15} are satisfied, so that the functor $\overre$ is automatically fully faithful. Now, $\overre$ has a section, namely the restriction of $c^*$ to $\rs\dashmod mod(\shr)$, so $\overre$ is essentially surjective. Consequently, $\overre$ is a Quillen equivalence.
\end{rem}

\subsection{The functor $\mot$} \label{sub:mot}

Now that we have identified $\shg$ as a category of modules in $\shr$ over the ring spectrum $\rs$, we are in a good situation to do descent for $\G$-spectra. More precisely, given a $\G$-spectrum $E$, we want to build the closest motivic spectrum to $E$. It turns out that the spectrum that suits our purposes is $c^*(E)^{\wedge}_{\sing(\S)}$. For simplicity, we will give a shorter name to this functor (see Definition \ref{de:mot}).

We first need some preliminary constructions. Essentially, we build a cosimplicial motivic spectrum from the pair of adjoints $(\re,\sing)$, and realize it.

\begin{de}
Let $\P : \shr \rightarrow \shr$ be the monad defined by $\P = \sing \re$,
$$ \eta : \S \rightarrow \sing \re$$
given by the unit of the adjoint pair $(\re,\sing)$, and
$$\mu : \P \P \rightarrow \P$$
being $\sing \epsilon_{\re}$, where $\epsilon$ is the counit of the adjoint pair $(\re,\sing)$.
\end{de}

\begin{rem}
First, observe that, for any $E \in \shg$, the motivic spectrum $\sing (E)$ has a natural $\P$-module structure, as it is $\sing \re c^* E$ by Proposition \ref{pro:secresing}.
\end{rem}

\begin{de}
Let $E\in \shg$. We define $\mot^{\bullet}(E) : \Delta^{op} \rightarrow \shr$ as the cobar construction associated to the monad $\P$, with respect to the algebra over it $\sing(E)$.

Precisely, for a spectrum $E \in \shg$, $\mot^n (E) = (\re \sing)^{n} \sing(E)$, and the faces and degeneracies are given by the unit and counit of the adjunction.
\end{de}

\begin{pro} \label{pro:contractiblecosimplicial}
Let $E \in \shg$. The coaugmented cosimplicial space $\re \mot^{\bullet}(E) \leftarrow E$ is contractible (\textit{i.e.} it is a coaugmented cosimplicial motivic spectrum  with extra degeneracies). 
\end{pro}

\begin{proof}
This is classical, the extra degeneracies comes from the $\re \sing$-algebra structure on $E$ (see \cite[Remark 4.5.3]{R14}).
\end{proof}

We are now ready to define the motivic spectrum $\mot(E)$.

\begin{de}\label{de:mot}
Let
$$\mot : \shg \rightarrow \shr$$
be the totalization of $\mot^{\bullet}$.
\end{de}

Explicitly, this object is
$$Tot \left( \xymatrix{ \hdots  \ar@<.4cm>[r] \ar[r] \ar@<-.4cm>[r]  & \sing \re \sing \re \sing E \ar@<.2cm>[r] \ar@<-.2cm>[r]  \ar@<.2cm>[l]  \ar@<-.2cm>[l] \ar@<.6cm>[l]  \ar@<-.6cm>[l]& \sing \re \sing E   \ar@<.4cm>[l] \ar[l] \ar@<-.4cm>[l] \ar[r]& \sing E \ar@<.2cm>[l] \ar@<-.2cm>[l] } \right).$$

\begin{rem} \label{rk:classicalrewriting}
The spectrum $\mot^{\bullet}(E)$ is the most natural candidate to build a motivic spectrum which is close to $E$. Note however the two following points: \begin{enumerate}
\item this construction seems to be relevant only in $\shr$, and useless in $\shc$. Indeed, the unit of the corresponding adjunction 
$$ \re^{\C} : \sh \leftrightarrows \shc : \sing^{\C}$$
over $\spec(\C)$ is a weak equivalence (this can be shown using the formal properties of the adjunction $(\re,\sing)$ over $\C$ in \cite{HO14}, and \cite[Proposition 3.0.2]{I14} implies that this unit coincides with tensoring with the unit of the ring spectrum $\S[\tau^{- 1}]$, which is idempotent). As a result, the analogous complex construction, say $\mot_{\C}^{\bullet}(E) =  (\re \sing)^{\bullet} \sing E$, is equivalent to $\sing E$ in degree zero.
\item the cosimplicial space $\re \mot^{\bullet}(E)$ is contractible by Proposition \ref{pro:contractiblecosimplicial}. However, since $\re$ does not commutes with homotopy limits, $\re \mot(E)$ tries to be $E$, but is not in general. We will see however that is is the case in some  particular situations.
\end{enumerate}
\end{rem}

We end this section by an easy result that ensures that our functor $\mot$ preserves ring structures.

\begin{pro} \label{pro:motandrings}
Let $E$ be a $\G$-equivariant $A_{\infty}$-ring spectrum (respectively $E_{\infty}$ ring spectrum). Then $\mot(E)$ is am $A_{\infty}$-ring spectrum (respectively $E_{\infty}$ ring spectrum).
\end{pro}

\begin{proof}
The limits in the categories of associative (respectively commutative) ring spectra are computed in the underlying category. Therefore, it suffices to interpret the totalization as being a limit in the appropriate category.
\end{proof}

\subsection{Motivic homology and the monad $\re\sing$} \label{sub:hfrandresing}

Our next objective is to compute the cohomology of $\mot E$ in terms of $\hmf_{*}(E)$. In the general case, the strategy is to study the Bousfield-Kan spectral sequence in motivic cohomology associated to the totalization $\mot(E)$. In some particular cases, we will see that this spectral sequence collapses at the second page, and gives a simpler computation of $\hfr_{*}(\mot(E))$. The main result is given in Theorem \ref{thm:cohmote}.

Observe that $\hmf_{*}$ is an $(\hfr)_{*}$-module.

\begin{lemma} \label{lemma:hfrofsinge}
Let $E$ be a $\G$-spectrum. There is a natural isomorphism of $(\hfr)_{*}$-modules
$$(\hfr)_{*}(\sing E) \cong \hmf_{*}(E).$$
\end{lemma}

\begin{proof}
First, recall that by definition 
$$(\hfr)_{*}(\sing E)  =  [S^{*}, \sing E \wedge \hfr]_{\R}.$$
Now, by the projection formula (Lemma \ref{lemma:projection}),
$$ [S^{*}, \sing (E) \wedge \hfr]_{\R} = [S^{*}, \sing (E \wedge \re \hfr)]_{\R},$$
which is $[S^{*}, \sing (E \wedge \hmf)]_{\R}$
by Proposition \ref{pro:rehfr}.

Using the adjoint pair $(\re, \sing)$, we have 
$$[S^{*}, \sing (E \wedge \hmf)]_{\R} = [S^{*}, E \wedge \hmf]_{\G},$$ as $\re(S^*) = S^*$.
The result follows.
\end{proof}

We do not only need to compute the cohomology of $\sing E$ in terms of the cohomology of $E$, but we need to do the same for all the stages of the cosimplicial spectrum $\mot^{\bullet}E$. This is where we take advantage of the description of the adjoint pair $(\re,\sing)$ provided by Theorem \ref{thm:qshinmot}.

\begin{lemma} \label{lemma:singreasrswedge}
Let $E \in \shg$. There is a natural weak equivalence
$$ (\sing \re)^n \sing E \cong \rs^n \wedge \oversing(E).$$
Moreover, the faces and degeneraces of $\mot^{\bullet}(E)$ are induced by the unit and multiplication of $\rs$, and the $\rs$-module structure of $\oversing(E)$ under this identification.
\end{lemma}

\begin{proof}
This is a direct consequence of Theorem \ref{thm:qshinmot}.
\end{proof}

In particular, we would like to use the K\"unneth spectral sequence to compute the homology of $ (\sing \re)^n \sing E $. Hopefully, we have the following result which almost gives a K\"unneth isomorphism.

Recall the notation $\M$ and $D\M$ adopted in \ref{nota:manddm}.

\begin{lemma} \label{lemma:torcomputation}
For all $i \geq 0$, $\tor_i^{\M}(D\M,D\M) = 0$.
In particular,
$$\tor_i^{(\hfr)_{*}}(\hmf_{*},\hmf_{*}) = 0$$
for all $i>0$.
\end{lemma}

\begin{proof}
We use an explicit flat resolution of $D\M$ of size $2$:
$$ K \inj \F[\rho^{\pm 1}, \tau^{\pm 1}] \surj D\M.$$

Observe that every element of $D\M$ is both $\rho$-torsion and $\tau$-torsion. Thus, for any $\M$-module $N$ whose elements are either $\rho$-divisible or $\tau$-divisible, $N\otimes_{\M} D\M = 0$. The three modules $K$, $\F[\rho^{\pm 1}, \tau^{\pm 1}]$, and $D\M$ satisfies this property. The result follows.
\end{proof}

\begin{cor}
Let $E\in \shg$ having free $\hmf$-cohomology. Then
$$\hfr_{*}((\rs)^{n} \wedge \sing(E)) \cong \left( \hmf_{*} \right)^{\otimes_{\hfr_{*}}n} \otimes_{\hmf_{*}} \hmf_{*}(E).$$
\end{cor}

\begin{proof}
By Lemma \ref{lemma:singreasrswedge}, there is a K\"unneth spectral sequence computing the desired cohomology. It collapses because of Lemma \ref{lemma:torcomputation} under the assumption that $\hmf_{*}(E)$ is a free module, because of Lemma \ref{lemma:hfrofsinge}.
\end{proof}

\subsection{The motivic homology of some spectra in the image of $\mot$} \label{sub:hfrmot}

We will now consider the Bousfield-Kan spectral sequence associated to the totalization $\mot(E)$. The previous subsection gave an explicit computation of the homology of the stages of this cosimplicial motivic spectrum, depending only on the $\G$-equivariant homology of $E$, so we essentially know the $E_2$-page of this spectral sequence.

\begin{rem}
The question of determining the (generalized) homology of a cosimplicial spectrum is hard in general. The only way to get it through is to have some finiteness assumption. This is the meaning of the extra-hypothesis on $E$ appearing in Theorem \ref{thm:cohmote}.
\end{rem}

\begin{pro}
The Bousfield-Kan spectral sequence computing $\hfr_{*}(\mot(\S))$ collapses at $E_2$. Moreover $\hfr_{*}(\mot(\S)) = \hfr_{*}.$
\end{pro}

\begin{proof}
This is done by inspection, the $n$th spectrum of this simplicial object has homology 
$$ \M \oplus \nor D\M^{\oplus n+1}.$$
Thus, the Bousfield-Kan spectral sequence converging to the homology of the simplicial spectrum has $E_2$-page
 $$ E_2^{t-n,*} \cong \left\{  \begin{matrix} \Sigma^{-n} \nor D\M^{\oplus n} \text{ if $n$ even} \\ \Sigma^{-n} M \oplus \nor D\M^{\oplus n} \text{ if $n$ odd.} \end{matrix} \right. $$

First of all, this spectral sequence collapses at $E_2$, since the complex associated to the cosimplicial abelian group $\hmf_{*}^{n+1}$ is exact.

To conclude, it suffices to show that the spectral sequence converges completely in the sense of \cite{GJ09}. We invoke \cite[Lemma VI.2.2.0]{GJ09}.  Indeed, the corresponding limit satisfies a Mittag-Leffler condition: let $(t,k)\in \Z^2$. The contribution of $E_1$ to this degree is the sum over $n$ of $\left( \nor D\M^{\oplus n}\right)_{t-n,k}$, when $n$ even, and $\Sigma^{-n} \left( M \oplus \nor D\M^{\oplus n} \right)_{t-n,k}$, when $n$ is odd. Both are zero when $n$ is big enough, as $\M$ and $D\M$ are finite dimensional on each row $(*,k)$.
\end{proof}

\begin{thm} \label{thm:cohmote}
Let $E \in \shg$ be a $\G$-spectrum whose equivariant homology is $\hmf_{*}$-free. Suppose moreover that for all $k \in \Z$,  $\hmf_{(*,k)}(E)$ is a finite dimensional $\F$-vector space. Then, there is a natural isomorphism of $\hfr_{*}$-modules
$$\hfr_{*}(\mot(E)) \cong \hfr_{*} \otimes_{\hmf_{*}} \hmf_{*}(E).$$
\end{thm}

Before going into the proof, let's take a look at a couple of particular examples of $\hmf_{*}$-modules satisfying the hypothesis of the theorem.

\begin{rem} \label{rk:finitnessrk}
For example, the hypothesis of the theorem are satisfied in the following two particular cases: \begin{itemize}
\item $\hmf_{*}(E)$ is a finitely generated free $\hmf_{*}$-module,
\item the generators of $\hmf_{*}(E)$ as an $\hmf_{*}$-module are concentrated  are in degree $(t,k)$ for $t \geq 0$, $\frac{k}{2}+1 \leq k \leq t$, and there are finitely many in each degree.
\end{itemize}
\end{rem}

Note that the equivariant Steenrod algebra itself, and any of its $\hmf_*$-free subalgebras satisfies the second condition of remark \ref{rk:finitnessrk}.

\begin{proof}[Proof of Theorem \ref{thm:cohmote}]
Apply motivic homology to the cosimplicial motivic spectrum $\mot(E)$. As $\hmf_{*}(E)$ is a free $\hmf_{*}$-module, the $E_1$-page $E^{s,t}_1(E)$ of the Bousfield-Kan spectral sequence computing $\hfr_{*}(Tot(\mot(E)))$ splits as $$E^{s,t}_1(E) \cong E^{s,t}_1(\S) \otimes_{\hmf_{*}} \hmf_{*}(E).$$ Although the previous splitting is only a splitting of $\hfr_{*}$-modules, it is compatible with the first differential by construction. Thus, it collapses to the desired $\hfr_{*}$-module.
\end{proof}

\subsection{Steenrod action on the image of $\mot$} \label{sub:steenrodmot}

Let $E \in \shg$ be a spectrum satisfying the hypothesis of Theorem \ref{thm:cohmote}. There is a natural action of the motivic Steenrod algebra on $\hfr_{*}(E)$, since this is the cohomology of a motivic spectrum. There is also an action of the equivariant Steenrod algebra on $\hmf_{*}(E)$. We will see that these two are closely related.

First, recall that by Proposition \ref{pro:rehfr}, there is an isomorphism $\overre(\hmf) \cong \rs \wedge \hfr$. This gives a map
$$ \hfr \wedge \hfr \rightarrow \rs \wedge \hfr \wedge \hfr \cong \overre(\hmf \wedge \hmf).$$
Since $\re$ is a monoidal functor, taking homotopy groups gives the following Hopf algebroid morphism.

\begin{de}
Betti realization gives a morphism of Hopf algebroids
$$\re(\A) : (\hfr_{*},\Ar) \rightarrow (\hmf_{*}, \Aq).$$
\end{de}

\begin{pro} \label{pro:reandstecomod}
The map $\re(\A)$ is induced by the inclusion of $\hmf_{*}$-modules 
$$\hfr_{*} \inj \hmf_{*}.$$

In particular, the associated functor
$$\re(\A) : \Ar\dashmod Comod \rightarrow \Aq\dashmod Comod$$ is given by extension of scalars $\hmf_{*}\otimes_{\hfr_{*}} (-)$.
\end{pro}

\begin{proof}
This is trivial for the "objects" of the Hopf algebroid. For the "morphisms", decompose $\hfr \wedge \hfr$ as a coproduct of $\hfr$ on both sides. The result then follows from the "objects" part, together with the fact that there exist an isomorphism $\Aq \cong \hmf_{*} \otimes_{\hfr_{*}} \Ar$.
\end{proof}

\begin{thm} \label{thm:hmote}
Let $E \in \shg$ be a $\G$-spectrum whose equivariant cohomology is $\hmf_{*}$-free. Suppose moreover that for all $k \in \Z$,  $\hmf_{(*,k)}(E)$ is a finite dimensional $\F$-vector space. Then
$$\hfr_{*}(\mot(E)) \cong \hfr_{*} \otimes_{\hmf_{*}} \hmf_{*}(E).$$
as modules over the motivic Steenrod algebra.
\end{thm}

\begin{proof}
Apply Betti realization to $\mot(E)$. For the same reason as in the proof of Theorem \ref{thm:cohmote}, the Bousfield-Kan spectral sequence computing $\hmf_{*}(\re(\mot(E)))$ collapses at $E_2$, and the spectral sequence converges.

Thus, the $\Aq$-comodule structure on $\hmf_{*}(\re(\mot(E)))$ can be identified in two ways: \begin{itemize} \item it is precisely $\hfr_{*}(\mot(E)) \cong \hfr_{*} \otimes_{\hmf_{*}} \hmf_{*}(E)$, by Proposition \ref{pro:reandstecomod},
\item it can also be obtained through the Bousfield-Kan spectral sequence computing $\hmf_{*}(Tot(\re\mot^{\bullet}(E)))$. The latter one is $ \hmf_{*}(E)$ by Proposition \ref{pro:contractiblecosimplicial}.
\end{itemize}
Naturality in $\hmf$ finishes the proof.
\end{proof}

Recall that there are $\G$-equivariant and motivic versions of the sub-algebras of the Steenrod algebra $\E(n)$ and $\A(n)$ (see \cite{Ric1} and \cite{Gre12}). Denote these $\Eq(n)$, $\Aq(n)$, and $\cEr(n)$, $\cAr(n)$. Recall from \cite{Ric1} and \cite{Gre12} that the quotient Hopf algebroids $\cAr\sur\cEr(n)$, $\cAr\sur\cAr(n)$, $\Aq\sur\Eq(n)$, and $\Aq\sur\Eq(n)$ are free as modules over the cohomology of a point.

\begin{cor}
Let $E \in \shg$ be an equivariant spectrum whose homology is $(\Aq \sur B)_*$, where $B$ is any of the algebras $\Eq(n)$, $\Aq(n)$, then the spectrum $\mot(E)$ has also homology $(\cAr \sur B_{\R})$, where $B_{\R}$ is the sub-Hopf-algebra of the movitic Steenrod algebra $\M \otimes_{\hmf_*} B$.
\end{cor}

\part{Equivariant and motivic topological modular forms} \label{part:mmf}

\section{The definition of equivariant modular forms} \label{sec:definitiontmf}

The tools developed in the first part of this paper reduce the construction of $\tmfr$ to an equivariant one. Indeed, by Theorem \ref{thm:cohmote}, it suffices to plug a $\G$-equivariant version of $\tmf$ in the functor $(c^*(-))^{\wedge}_{\sing(\S)}$ to produce a motivic spectrum with the desired cohomology as a module over the motivic Steenrod algebra. This part is devoted to the construction of this equivariant version of $\tmf$.

The main piece of technology we are using here is the Tate diagram, and apart from the general theory, we will focus on the particular case when $E = \tmf$. In that case, we want to use the spectrum $\tmf$ to produce $\tmfq$, an equivariant refinement of $\tmf$ whose cohomology is $\Aq\sur \Aq(2)$.

The author is indebted to Mike Hill for suggesting the construction of $\tmfq$ from its Tate diagram.

\subsection{The Tate diagram} \label{sub:tatediag}
The ultimate goal here is to construct a $\G$-spectrum $\tmfq$ which is a $\G$-equivariant refinement of $\tmf$. In other words, we want the underlying non equivariant spectrum and the fixed points of this $\tmfq$ to have a prescribed homotopy type. We thus start this section by a well-known tool to analyse a $\G$-equivariant spectrum from non-equivariant data: the Tate diagram.

Let $\univ$ be a universal $\G$-space, that is a contractible free $\G$-space. Such a $\G$-space is unique up to $\G$-equivariant homotopy. Although we will not directly use it, it is good to know that a space that satisfies these properties is the unit sphere in $\infty\sign$, so we can take $\univ = S(\infty\sign)$. Let $\tuniv$ be the cofiber of the map $\univ_+ \rightarrow S^0$, which sends $\univ$ to the non-base-point. Again, a possible model for it is $S^{\infty\sign}$. A consequence of the natural filtration of this sphere by $S^{n\sign}$ plays a role  in Mahowald's model for the Tate spectrum, which in turn is a crucial ingredient in \cite{BR17}.

Let $E \shg$. The Tate diagram of $E$ is the following commutative diagram, where the rows are cofiber sequences:

\begin{equation} \label{eq:tate}
\xymatrix{ \univ_+ \wedge E \ar[r] \ar[d] & E \ar[r] \ar[d] & \tuniv \wedge E \ar[d] \\
\univ_+ \wedge F(\univ_+, E) \ar[r] & F(\univ_+,E) \ar[r] & \tuniv \wedge F(\univ_+,E)}
\end{equation}

Note that by Greenlees-May \cite[Section I.1]{GM03}, the leftmost vertical arrow is a weak equivalence. In particular, the rightmost square is a homotopy pullback.

The spectra that appear in the Tate diagram have alternative name, which are more convenient to use:

\begin{de}
We rewrite the Tate diagram of $E$ \eqref{eq:tate} using the usual notations for the spectra appearing in it:
\begin{equation*}
\xymatrix{ \univ_+ \wedge E \ar[r] \ar[d]^{\cong} & E \ar[r] \ar[d] & \tuniv \wedge E \ar[d] \\
f(E) \ar[r] & c(E) \ar[r] & t(E).}
\end{equation*}
The spectrum $t(E)^{\G}$, where $(-)^{\G}$ denotes the fixed points functor, is called the Tate spectrum of $E$.
\end{de}

Recall that Lewis defined in \cite{Le95} a change of universe functor. This gives a pushforward functor $i_* : \sh \rightarrow \shg$, which sends a non-equivariant spectrum $X$ to the extension to the complete universe of the naive $\G$-equivariant spectrum $X$, viewed as a spectrum with trivial $\G$-action (see \textit{loc cit} for universes and change of universe functors).

\begin{pro} \label{pro:tatering} \label{pro:feqring}
Let $X$ be a non-equivariant ring spectrum. Denote again by $X$ the pushforward of $X$ in the category of $\G$-spectra. Then $t(X)$, $c(X)$ and $\tuniv\wedge X$  are ring $\G$-spectra. Moreover, the following square is a pullback of ring $\G$-spectra:
\begin{equation}
\xymatrix{ X \ar[r] \ar[d] & \tuniv\wedge X \ar[d] \\ c(X) \ar[r] & t(X)}
\end{equation}
\end{pro}

\begin{proof}
See \cite[Proposition 3.5]{GM03} applied to the $\G$-spectrum $X$ (denoted $i_*X$ in \textit{loc cit}) gives the result.
\end{proof}

\subsection{The definition of $\tmfq$} \label{sub:tmfq}

Let $n \in \N$ and $E \in \shg$. We denote by $E[x^n]$ the spectrum $\bigvee_{i=0}^{\infty} \Sigma^{ni} E$. We also denote by $E((x^n))$ the infinite product $\prod_{i=-\infty}^{\infty} \Sigma^{ni}E$. Note that such constructions appeared in \cite{AMS98}.
Furthermore, when $E$ is a connective spectrum, there is a weak equivalence
\begin{equation*}
\lim_i \bigvee_{k = -i}^{\infty} \Sigma^{kn}E \cong E((x^n)).
\end{equation*}

Finally, we recall from \cite{AMS98} that, whenever $E$ is a ring, one can define a multiplication on $E((x^n))$ which is compatible with the ring structure of $E$.

We start this subsection by the main result of \cite{BR17}, which is an essential ingredient here.

\begin{thm*}[{\cite[Theorem 1.1]{BR17}}]
There is a weak equivalence of spectra
$$t(\tmf)^{\G} \cong \ko((x^8))$$
where $x^8$ is a formal element in degree $8$.
\end{thm*}

Note that there is not any element $x$. We hope that this does not generate too much confusion for now. The notation will be justified later, since $x^8$ comes from an equivariant class, which is an $8$th power (see Lemma \ref{lemma:definephi}).

In particular, there is a weak equivalence of $\G$-spectra
\begin{equation}
t(\tmf) \cong \tuniv \wedge \ko((x^8)),
\end{equation}
where $\ko((x^8))$ is viewed as a $\G$-equivariant spectrum through the pushforward.

\begin{de} \label{de:tmfq}
Let $\tmfq$ be the homotopy pullback
\begin{equation} \label{eq:pullbacktmfq}
\xymatrix{ \tmfq \ar[d] \ar[r] & \tuniv \wedge \ko[x^8] \ar@{^(->}[d] \\
 F(E{\G}_+,\tmf) \ar[r] & \tuniv \wedge \ko((x^8))}
\end{equation}
where the rightmost arrow is the inclusion. 
\end{de}

At this point of the discussion, the reader might be confused by the appearance of the spectrum $\tuniv \wedge \ko[x^8]$ which has not been motivated yet.
The choice of this spectrum is explained in two ways: first, this is the spectrum that should be put here in order to have the correct cohomology for $\tmfq$ (see the computation we make in section \ref{sec:cohomlogyoftmfq}). Moreover, this is the most natural guess, by analogy with the chromatically lowest  analogues known in $\G$-equivariant stable homotopy theory. Indeed, a direct consequence of the computations in \cite{HK01} is that the Tate diagram for $\hmf$ and $\hmz$ are the following pullback diagrams:

\begin{equation*}
\xymatrix{ \hmz \ar[r] \ar[d] & \tuniv \wedge H\F[x^2] \ar[d] \\
F(\univ_+ , H\Z) \ar[r] & \tuniv \wedge H\F((x^2)) }
\end{equation*}

and 

\begin{equation*}
\xymatrix{ \hmf \ar[r] \ar[d] & \tuniv \wedge H\F[x] \ar[d] \\
F(\univ_+ , H\F) \ar[r] & \tuniv \wedge H\F((x)). }
\end{equation*}

We end this section by a conjecture, on the structure of $\tmfq$. Observe that the bottom row of the diagram  \eqref{eq:pullbacktmfq} is a morphism of ring spectra as $F(E{\G}_+,\tmf)$ is a ring by Proposition \ref{pro:feqring} since $\tmf$ is one, and $t(\tmf)$ is a naive ring by Proposition \ref{pro:tatering}, and the map between them is a map of rings.

Moreover, the ring structure on the non-equivariant spectrum $\ko[x^8]$ gives rise to a natural ring structure on the $C_2$-equivariant spectrum  $\tuniv \wedge \ko[x^8]$, which is at the top right corner of \eqref{eq:pullbacktmfq}. If we knew that the rightmost map in this diagram was a map of ring spectra, then we could define $\tmfq$ as the pullback \emph{in the category of ring spectra}. Since the forgetful functor from ring spectra to spectra creates limits, this would give a ring structure on $\tmfq$.

\begin{conj} \label{pro:tmfqring}
The $\G$-spectrum $\tmfq$ is a $\G$-ring spectrum (in the naive sense).
\end{conj}

\section{Cohomology of $\tmfq$} \label{sec:cohomlogyoftmfq}

The main result of this section is Theorem \ref{thm:cohtmfq}, which computes the $\G$-equivariant cohomology of $\tmfq$. Note that the spectrum $\tmfq$ being build from known non -equivariant spectra using the Tate diagram, the determination of $\hmf^*(\tmfq)$ is essentially a computation of the $\G$-equivariant cohomology of non-equivariant spectra (with a trivial action of $\G$).

\begin{thm} \label{thm:cohtmfq}
There is an isomorphism of $\Aq$-modules
$$\hmf^{*}(\tmfq) \cong \Aq\sur\Aq(2).$$ 
\end{thm}

The proof of this theorem is an inspection on the defining pullback square of Definition \ref{de:tmfq}. We start by identifying the upper right corner in terms of the dual equivariant Steenrod algebra.

To study of the cohomology of any spectrum by its Tate diagram we should first understand the equivariant cohomology of $\tuniv \wedge X$ in terms of the non equivariant cohomology of $X$.

Let's first determine what happens at the level of the dual Steenrod algebras. 

\begin{nota}
Let $\zeta_i$ and $\theta_i$ be the conjugate of the elements $\xi_i$ and $\tau_i$ of the $\G$-equivariant dual Steenrod algebra.
Similarly, $\pz_i$ will denote the conjugate of $\pxi_i$ in the classical dual Steenrod algebra.
\end{nota}

\begin{lemma}\label{lemma:definephi}
Smashing with $\tuniv$ induces a map of geometric Hopf algebroids $$\hmf \wedge \hmf \rightarrow \tuniv \wedge H\F \wedge H\F[x,y].$$
In homotopy, this is the Hopf algebroid map 
$$ \Phi_{\A} :\hmf_{*}[\theta_i,\zeta_{i+1}]/(\theta_i^2 + \rho \theta_{i+1} +  \tau \zeta_{i+1}) \rightarrow \hmf_{*}[\rho^{- 1},\tau,x][\pz_i]$$
defined on generators by $\Phi_{\A}(\theta_0) = x$, and
\begin{equation} \Phi_{\A}(\zeta_i) = \rho^{2^i-1}\pz_i + \sum_{i=1}^{n} \rho^{2^n - 2^i} \eta_R(\tau)^{2^{i-1}-1} \xi_{n-i}^{2^{i-1}} \tau_{n-1}.
\end{equation} Here, $x := \eta_R(\rho^{-1} \tau) +\rho^{-1} \tau$.
\end{lemma}

\begin{proof}
This is precisely the general formula given in Theorem \ref{thm:comp} after $\rho$ is inverted.
\end{proof}

\begin{lemma} \label{lemma:coactiongeometricfixed}
Let $X \in \sh$. There is an isomorphism of $\hmf_{*}$-modules
$$ \hmf_{*}(\tuniv \wedge X) \cong \F[\tau,\rho^{\pm1}] \otimes_{\F} H\F^*(X).$$
Moreover, the coaction of the dual equivariant Steenrod algebra is entirely determined by the map
$$ \Phi_{\A} : \Aq[\rho^{- 1}] \rightarrow (\Phi^\G( \hmf \wedge \hmf))_{*} = \F[\tau,\rho^{\pm 1},x] \otimes_{\F} \A.$$
\end{lemma}

\begin{proof}
This comes from Theorem \ref{thm:twistedext} after $\rho$ is inverted.
\end{proof}

\begin{pro} \label{pro:koxeight}
As a module over the equivariant Steenrod algebra, $$\hmf_{*}(\ko[x^8]) \cong \hmf_{*}[\rho^{-1}][\theta_3,\theta_4,\hdots, \zeta_1^4, \zeta_2^2, \zeta_3,\hdots]/(\theta_i^2 + \rho \theta_{i+1} +  \tau \zeta_{i+1}),$$
whose dual is $\Aq\sur\Aq(2)[\rho^{- 1}]$.
\end{pro}

\begin{proof}
This is a formal consequence of Lemma \ref{lemma:coactiongeometricfixed} and the explicit formulae for $\Phi_{\A}$ given in Lemma \ref{lemma:definephi}, as the leading term in the polynomial $\comp(\pxi_i)$ is $\rho^{2^i-1} \xi_i$.
\end{proof}

We now turn to the other terms of the Tate diagram. It turns out the easiest way to do it is to compute the fiber of the bottom row in the Tate pullback, because of the following fact.

\begin{lemma}\label{lemma:eqtmf}
The homotopy fiber of the map
$$ F(\univ_+, \tmf) \rightarrow \ko((x^8))$$
is $\univ_+\wedge \tmf$.
\end{lemma}

\begin{proof}
This is by definition of our map $ F(\univ_+, \tmf) \rightarrow \ko((x^8))$ (this is the same as the one appearing in the Tate diagram for $\tmf$ together with a trivial $\G$-action.
\end{proof}

\begin{lemma} \label{lemma:eqx}
There is an isomorphism of comodules over the dual equivariant Steenrod algebra
$$ \hmf_{*}(\univ_+ \wedge \tmf) \cong (\univ_+ \wedge \hmf)_{*}[\theta_3,\theta_4,\hdots, \zeta_1^4, \zeta_2^2, \zeta_3,\hdots]/(\theta_i^2 + \rho \theta_{i+1} +  \tau \zeta_{i+1}).$$
\end{lemma}

\begin{proof}
This is entirely analogous to the $\tuniv \wedge (-)$ part. Here, the morphism of geometric Hopf algebroids
$$ \univ_+ \wedge \hmf \wedge \hmf \cong (\univ_+ \wedge \hmf)_{*}[ \pz_1,\pz_2, \hdots ]$$
is the morphism appearing in \cite[p.385]{HK01}. Namely, the two sides are $(\univ_+ \wedge \hmf)_{*}$-Hopf algebroids, free as $(\univ_+ \wedge \hmf)_{*}$-modules, and with algebra generators in the same distinct degrees.
From this, it is clear that $\tau_0$ is sent to $\pz_1$, and that the $\zeta_i$ are sent to $\pz_i^2$ modulo terms of lower topological degree. The compatibility with the diagonal concludes the identification.
\end{proof}

To conclude the proof of Theorem \ref{thm:cohtmfq}, we have to actually identify the extension
$$ \univ_+ \wedge \tmfq \rightarrow \tmfq \rightarrow \tuniv \wedge \tmfq$$
in equivariant homology. To this end, we use the following observation: 

\begin{lemma} \label{lemma:tmfqorientation}
The Tate diagram for $\tmfq$ maps to the Tate diagram for $\hmf$. In particular, this provides a map $\tmfq \rightarrow \hmf$.
\end{lemma}

\begin{proof}
The map $\tmf \rightarrow H\F$, already gives a morphism of cofiber sequences 
$$ \xymatrix{ \univ_+ \wedge F( \univ_+, \tmfq) \ar[r] \ar[d] & F( \univ_+, \tmfq) \ar[r] \ar[d] & \tuniv \wedge F( \univ_+, \tmfq) \ar[d] \\ \univ_+ \wedge F( \univ_+, \hmf) \ar[r] & F( \univ_+, \hmf) \ar[r] & \tuniv \wedge F( \univ_+, \hmf) }
$$
since the underlying spectrum of $\hmf$ is $H\F$, and that the decomposition \linebreak $t(\tmf) \cong \ko((x^8))$ is compatible with $t(H\F) \cong H\F((x))$.

To conclude, note that the following square is commutative
$$ \xymatrix{ \ko[x^8] \ar[r] \ar[d] & H\F[x] \ar[d] \\\ko((x^8)) \ar[r] & H\F((x)), } $$
where the horizontal arrows are induced by $\ko \rightarrow H\F$.
\end{proof}

\begin{proof}[Proof of Theorem \ref{thm:cohtmfq}]
By Proposition  \ref{pro:koxeight}, we know the coaction of the Steenrod algebra on $\hmf(\tuniv\wedge \tmfq)$,  and by Lemmas \ref{lemma:eqtmf} and \ref{lemma:eqx}, we also know that \linebreak $\hmf_*( \univ_+ \wedge \tmfq) \cong \hmf_*(\univ_+) \otimes_{\hmf_*} \left( \Aq\sur\Aq(2)\right)^{\vee}$.

Now, Lemma \ref{lemma:tmfqorientation} gives a map of long exact sequences
\begin{equation*}
\xymatrix{
\hmf_*(\univ_+) \otimes_{\hmf_*} \left( \Aq\sur\Aq(2) \right)_* \ar[r] \ar[d]& \hmf_*(\tmfq) \ar[r] \ar[d]& \left( \Aq\sur\Aq(2)\right)_*[\rho^{-1}] \ar[d] \\
\hmf_*(\univ_+) \otimes_{\hmf_*} \Aq_* \ar[r]& \Aq_* \ar[r]& \Aq_*[\rho^{- 1}]. }
\end{equation*}

Moreover, the leftmost and rightmost vertical arrows induce the previous isomorphisms
\begin{equation}
\hmf_*( \univ_+ \wedge \tmfq) \cong \hmf_*(\univ_+) \otimes_{\hmf_*} \left( \Aq\sur\Aq(2)\right)^{\vee}
\end{equation}
and 
\begin{equation}
\hmf(\tuniv\wedge \tmfq) \cong  \left( \Aq\sur\Aq(2)\right)^{\vee}[\rho^{- 1}].
\end{equation}

This gives the desired isomohphism $\hmf_*(\tmfq) = \left( \Aq\sur\Aq(2)\right)_*$ by the $5$-lemma.
\end{proof}

\section{Motivic versions} \label{sec:mmf}

\subsection{Motivic modular forms over $\spec(\R)$} \label{sub:tmfr}

\begin{de} \label{de:mmfr}
Let $\tmfr$ be $\mot(\tmfq)$.
\end{de}

\begin{thm} \label{thm:tmfr}
There is an isomorphism of $\Ar$-modules
$$\chfr^{*}(\tmfr) \cong \cAr \sur \cAr(2).$$
\end{thm}

\begin{proof}
This is a consequence of Theorem \ref{thm:hmote} since the hypothesis is satisfied for $\tmfq$ by Theorem \ref{thm:cohtmfq}.
\end{proof}

%
%\section{Ring structure}
%
%IN PROGRESS.

\subsection{Motivic modular forms over $\spec(\C)$ and additional comments} \label{sub:tmfc}

As observed in Remark \ref{rk:steenrodalgebras}, the corresponding descent technique developed in Part \ref{part:comparison} does not provide a motivic spectrum over $\spec(\C)$ whose cohomology is $\cAc\sur\cAc(2)$.

However, the pullback of $\tmfr \in \shr$ to $\shc$ is better behaved.

\begin{de} \label{de:demmfc}
Let $\tmfc$ be $p_* \mot(\tmfq)$.
\end{de}

We now compute the cohomology of this spectrum $\tmfc$. To do so, we need a result about the relationship between the functor $p^*$ and motivic cohomology.

\begin{lemma} \label{lemma:comparisonwithc}
There is a weak equivalence 
\begin{equation}
p_*\hfc \cong \frac{\hfr}{(\rho)}.
\end{equation}
In particular, if the cohomology of $X \in \shr$ is free over the coefficient ring, then
\begin{equation}
\chfc^*(p^*X) \cong \left(p_*\chfc\right)^*(X) \cong \frac{\chfr^*(X)}{(\rho)}.
\end{equation}
\end{lemma}

\begin{proof}
By construction, the Eilenberg-MacLane motivic spaces are compatible with the functor $p^*$ (see the construction in \cite{Voe03b}). Consequently,
\begin{equation}
p^*\hfr \cong \hfc.
\end{equation}

Consider the unit of the adjunction $(p^*,p_*)$:
\begin{equation*}
\hfr \rightarrow p_*p^* \hfr.
\end{equation*}

Using the weak equivalence 
\begin{equation*}
p^*\hfr \cong \hfc,
\end{equation*}
we obtain
\begin{equation*}
f: \hfr \rightarrow p_*p^* \hfr \cong p_*\hfc.
\end{equation*}
We want to show that $\rho f$ is nullhomotopic. The composite map $\rho f : \Sigma^{-\sign} \hfr \rightarrow p_* \hfc$ is adjoint to a morphism $\Sigma^{-\sign} p^* \hfr \rightarrow p^* \hfr$, since $p^* S^{-\sign} = S^{-\sign}$. This is zero since there is no element in $\Ac$ in this degree.

This produces a morphism $g : \frac{\hfr}{(\rho)} \rightarrow p_* \hfc$. This morphism is an isomorphism in homotopy groups, since
\begin{equation}
\pi^{\R}_{*}\left(\frac{\hfr}{(\rho)}\right) \cong \pi^{\C}_{*}(\hfc) \cong \pi^{\R}_{*}(p_*\hfc),
\end{equation}
using again that $p^*S^{*} \cong S^{*}$.
\end{proof}

\begin{cor} \label{cor:cohommmfc}
There is an isomorphism of $\cAc$-modules
$$\chfc^{*}(\tmfc) \cong \cAc \sur \cAc(2) .$$
\end{cor}

\begin{proof}
This is an immediate consequence of Theorem \ref{thm:tmfr}, using Lemma \ref{lemma:comparisonwithc} to drag the result to $\spec(\C)$.
\end{proof}

\begin{rem}
The slight detour we had to take to build a version of $\tmf$ over $\spec(\C)$ is retrospectively clear: to build the equivariant version, we have a powerful tool we do not possess in the motivic setting, namely the Tate diagram.
Then, the proximity between the $\G$-equivariant Steenrod algebra and the  motivic Steenrod algebra over $\spec(\R)$ is an indication that the hard work is almost done.

Finally, the category $\shr$ being the top left corner of the commutative square
$$\xymatrix{ \shr \ar[r]^{\re} \ar[d]^{p^*} & \shg \ar[d]^{i^*} \\
\shc \ar[r]^{\re} & \sh}$$
of forgetfull monoidal functors, we have a version of topological modular forms in the most structured category of this square, and every other one is obtained by forgetting $\tmfr$.
\end{rem}

\begin{appendices}

\section{Non-equivariant and $\G$-equivariant (co)homology}

Let $Y$ be a $\G$-spectrum. Then, there is a natural $\Aq^*$-module structure on $\hmf^*(Y)$: this is a morphism of $\hmf^*$-modules
\begin{equation} \label{eq:nataction}
\Aq^* \otimes \hmf^*(Y) \stackrel{\mu}{\longrightarrow} \hmf^*(Y),
\end{equation}
where the tensor product uses the left $\hmf^*$-module structure on $\hmf^*(Y)$ and the right $\hmf^*$-module structure on $\Aq^*$.

Note that there is also a coaction of $\Aq_*$ on $\hmf_*(Y)$: this is the morphism of $\hmf_*$-modules
\begin{equation} \label{eq:natcoaction}
\lambdaq_* : \hmf_*(Y) \longrightarrow \hmf_*(Y) \otimes \Aq_*,
\end{equation}
where the tensor product is with respect to the left $\hmf_*$-module structure of both $\hmf_*(Y)$ and $\Aq_*$.
This coaction is induced by
\begin{align*}
\hmf \wedge Y  \xrightarrow{\hmf \wedge \eta \wedge Y} & \hmf \wedge \hmf \wedge Y \\
  \cong & \hmf \wedge \hmf \wedge_{\hmf} \hmf \wedge Y
\end{align*}
in homotopy, where $\eta : \S \rightarrow \hmf$ is the unit.

Our ultimate goal in this notes is to understand the $\Aq^*$-module structure of $\hmf^*(i_*X)$, where $X$ is a non-equivariant spectrum, using only the knowledge of $H\F^*(X)$ as an  $\A^*$-module.

It turns out that the $\hmf^*$-module $\hmf^*(i_*X)$ is always $\hmf^*$-free (see Lemma \ref{lemma:extensionofscalars}). In the case when $Y\in \shg$ has a flat (co)homology, the classical analysis Milnor did in the non-equivariant case (see \cite{Mi58}) generalizes (see \cite[Theorem 11.13]{Bo95}), so that it is equivalent to study either the action of $\Aq^*$ in cohomology provided by \eqref{eq:nataction} or the coaction of the dual Steenrod algebra $\Aq_*$ on $\hmf^*(i_*X)$, given in \eqref{eq:natcoaction}. 

In turn, the $\Aq_*$-comodule structure on $\hmf_*(Y)$ correspond to a coaction
\begin{equation} \label{eq:natcoactioncoh}
\lambdaq^* : \hmf^*(Y) \longrightarrow \hmf^*(Y) \otimes \Aq_*,
\end{equation}
by duality (see the general statement about the duality between homology and cohomology in \cite[Theorem 4.14]{Bo95}, for example).

The study of this coaction is the subject of the next subsection.

\section{Cartan formula and its dual}

For any spectrum $X \in \sh$, the $\hmf$-cohomology of its pushforward $\hmf^*(X)$ is a free $\hmf^*$-module (see Lemma \ref{lemma:extensionofscalars}). Therefore, a flatness or even freeness assumption on equivariant cohomology groups are harmless for our purposes.

Let $X$ and $Y$ be $\G$-spectra whose cohomology is free over the cohomology of a point. In particular, there is a K\"unneth isomorphism
\begin{equation*}
\hmf_*(X \wedge Y) \cong \hmf_*(X) \otimes \hmf_*(Y),
\end{equation*}
where the tensor product is taken over $\hmf_*$.

Then, the coaction of $\Aq_*$ on $\hmf_*(X) \otimes \hmf_*(Y) $ is given by the composite
\begin{equation} \label{eq:coactionproduct}
\hmf_*(X) \otimes \hmf_*(Y) \rightarrow \Aq_* \otimes \hmf_*(X) \otimes  \Aq_* \otimes \hmf_*(X) \rightarrow \Aq_* \otimes \hmf_*(X) \otimes \hmf_*(Y),
\end{equation}
where the last map is induced by the product on $\Aq_*$.

In particular, take $X \in \shg$ to be the suspension spectrum of an equivariant \emph{space}. Then the diagonal induces a coalgebra structure on $\hmf_*(X)$. The argument of the last paragraph gives in that case the following result.

\begin{pro}
Let $X \in \shg$ be a suspension spectrum. Then, the coaction $\lambdaq^*$ of the dual Steenrod algebra $\Aq_*$ on  $\hmf^*(X)$ is a morphism of algebra.
\end{pro}

\begin{proof}
This is a consequence of the naturality of $\lambdaq^*$, as defined in equation \eqref{eq:natcoactioncoh} for the morphism of $\G$-equivariant spectra $\Delta : X \rightarrow X \wedge X$, together with equation \eqref{eq:coactionproduct}.
\end{proof}

\begin{ex}
As an example, \cite{HK01} completely determine the map $\lambdaq^*$ for a particular $\G$-space. We recall and complete the analysis made in \textit{loc cit} here for completeness.

Let $B'\Z/2$ be the $\G$-space of lines in $\C^{\infty}$, together with the action induced by complex conjugation on $\C^{\infty}$. Note that there is a map
\begin{equation} \label{eq:iota}
\iota : i_*B\Z/2 \inj B'\Z/2,
\end{equation}
which is a non-equivariant weak equivalence.

In \cite[p.381]{HK01}, Hu and Kriz determine the equivariant cohomology of this space: there is a $\hmf^*$-module isomorphism
\begin{equation} \label{eq:bprimecohomology}
\hmf^*(B'\Z/2) \cong \frac{\hmf^*[c,b]}{(c^2 + \rho c + \tau b)}.
\end{equation}

The same formal argument as in \cite{Mi58} gives
\begin{equation*}
\lambdaq^*(c) = c \otimes 1 + \sum_{i \geq 0} b^{2^i} \otimes \tau_i. 
\end{equation*}

In turns, this gives 
\begin{equation*}
\lambdaq^*(b) = \sum_{i \geq 0} b^{2^i} \otimes \xi_i,
\end{equation*}
using $c^2 = \rho c + \tau b$ from \eqref{eq:bprimecohomology} and $\tau_i^2= \rho \tau_{i+1} + \eta_R(\tau) \xi_{i+1}$ from Proposition \ref{pro:steenrod}. This immediately gives the map $\lambdaq^*$ for any element of $\hmf^*(B'\Z/2)$.
\end{ex}

\section{Twisted extension of scalars}

Before we even think about the $\Aq_*$-comodule structure on $\hmf_*(i_*X)$, we need to understand its $\hmf_*$-module structure. This is the subject of the following lemma.

\begin{lemma} \label{lemma:extensionofscalars}
Let $X \in \sh$. There is an isomorphism of $RO(\G)$-graded $\hmf^*$-modules
\begin{equation*}
\hmf_*(i_*X)  \cong \hmf_* \otimes_{\F} H\F_*(X),
\end{equation*}
where the tensor product is the $RO(\G)$-graded tensor product over $\F$, and $H\F_*(X)$ is concentrated in degrees $\Z \subset RO(\G)$.
\end{lemma}

\begin{proof}
There is an evident morphism of $\F$-vector spaces $H\F^*(X) \rightarrow \hmf^*(i_*X)$. This extends to a natural  $\hmf^*$-module morphism $\hmf^* \otimes_{\F} H\F^*(X) \rightarrow \hmf^*(i_*X)$. Now, this morphism is a natural transformation of
non-equivariant cohomology theories which is an isomorphism for $X = S^0$.
\end{proof}

Remark that, by definition of $\hmf$, there is a weak equivalence $(\hmf)^\G \cong H\F$. Thus, by adjunction, there is a ring homomorphism
\begin{equation*}
\epsilon : i_*H\F \rightarrow \hmf.
\end{equation*}

This map induces a $\hmf^*$-module morphism in cohomology
\begin{equation*}
\tilde{\comp} : \hmf^* \otimes \A_* \rightarrow \Aq_*.
\end{equation*}

\begin{de} \label{de:phi}
Let $\comp : \A_* \rightarrow \Aq_*$ be the restriction of $\tilde{\comp}$.
\end{de}

We are now ready to define the twisted version of the extension of scalars functor defined in Lemma \ref{lemma:extensionofscalars} we need to understand $\lambdaq^*$ on non-equivariant spaces.

\begin{de}
Let $\hmf \ootimes (-) : \Acomod \rightarrow \Aqcomod$, defined, for any $\A_*$-module $M$, by $\hmf \ootimes M := \hmf^* \otimes_{\F} M$ as an $\hmf^*$-module, with the coaction map defined as follows
\begin{align*}
\lambdaq^* : \hmf^* \otimes_{\F} M \xrightarrow{\lambda^*} & \hmf^* \otimes_{\F} \A_* \otimes_{\F} M \\
\xrightarrow{\tilde{\comp}} & \Aq_* \otimes_{\F} M \\
\xrightarrow{\cong} & \Aq_* \otimes \hmf^* \otimes_{\F} M,
\end{align*}
where $\lambda^*$ is the non-equivariant coaction map.
\end{de}

This functor $\hmf^* \ootimes$ exactly express the effect if $i_*$ in cohomology, as we will see in the next Theorem. Note that this interpretation still lacks an explicit description of $\comp$ to be effective in computations. We defer this determination to subsection \ref{sub:formulacomp}.

\begin{thm} \label{thm:twistedext}
Let $X \in \sh$. There is an isomorphism of $\Aq_*$-comodules
\begin{equation*}
\hmf_*(i_*X)  \cong \hmf_* \ootimes H\F_*(X),
\end{equation*}
where $\ootimes$ is the twisted extension of scalars $\hmf \ootimes (-) : \Acomod \rightarrow \Aqcomod$.
\end{thm}

\begin{proof}
Recall that the coaction maps $\lambdaq^*$ (respectively its non-equivariant analogue $\lambda^*$) on the homology of $X$ is defined by inserting the unit of $\hmf$ (respectively $H\F$) appropriately in $\hmf \wedge X$ (respectively $H\F \wedge X$).

The result follows from the commutativity of the following diagram:

$$ \xymatrix{ \hmf \wedge_{i_*H\F} i_*(H\F \wedge X) \ar[r]^{\epsilon \wedge i_*X} \ar[d]^{ \hmf \wedge_{i_*H\F}i_*(H\F \wedge \eta \wedge X)}  &   \hmf \wedge i_*X  \ar[d]_{\hmf \wedge \eta \wedge i_*X} \\
 \hmf \wedge_{i_*H\F}i_*(H\F \wedge H\F \wedge X) \ar[r]^{\epsilon \wedge \epsilon \wedge i_*X} \ar[d]^{\cong \quad\quad} & \hmf \wedge \hmf \wedge i_*X  \ar[d]_{\cong \quad\quad} \\
 i_*\left( (H\F \wedge H\F) \wedge_{H\F} (H\F \wedge X) \right)  \ar[r]^{\epsilon \wedge \epsilon \wedge i_*X} & (\hmf \wedge \hmf) \wedge_{\hmf} (\hmf \wedge i_*X) , } $$

where the horizontal maps are isomorphisms.
\end{proof}

\section{The formula for $\comp$} \label{sub:formulacomp}

In this subsection, we give a formula for the morphism $\comp$. Note that, together with Theorem \ref{thm:twistedext}, this will give a closed formula for the coaction of $\Aq_*$ on the $\G$-equivariant cohomology of non-equivariant spaces.
The statement is given in Theorem \ref{thm:comp}.

The first step in this analysis is to find a space whose cohomology has a coaction of $\A_*$ that makes appear every generator $\xi_i \in \A_*$, and compare it to an equivariant space whose equivariant coaction is well known.

A non-equivariant space satisfying our needs is obviously the infinite projective space (see \cite{Mi58}).
\begin{pro}[{\cite{Mi58}}] \label{pro:hfbz2}
There is an isomorphism of algebras
\begin{equation*}
H\F^*(B\Z/2) \cong \F[x],
\end{equation*}
where $x$ is in degree $1$. Moreover, the algebra map $\lambda^*$ is entirely determined by
\begin{equation*}
\lambda^*(x) = \sum_{i \geq 0} x^{2^i} \otimes \pxi_i.
\end{equation*}
\end{pro}

In particular, Lemma \ref{lemma:extensionofscalars} gives that $\hmf^*(i_*B\Z/2) \cong \hmf^*[x]$. We will now compare $B\Z/2$ to its better-behaved $\G$-equivariant analogue $B'\Z/2$.

\begin{lemma}
The algebra morphism of $\hmf^*$-modules 
\begin{equation*}
\iota^* : \frac{\hmf^*[c,b]}{(c^2 + \rho c + \tau b)} \rightarrow \hmf^*[x]
\end{equation*}
induced by $\iota$  (provided by equation \eqref{eq:iota}) is determined by 
\begin{align*}
\iota^*(c) = \tau x \\
\iota^*(b) = \tau x^2 + \rho x.
\end{align*}
\end{lemma}

\begin{proof}
The isomorphism 
\begin{equation*}
\hmf^*(B'\Z/2) \cong \frac{\hmf^*[c,b]}{(c^2 + \rho c + \tau b)}
\end{equation*}
is given in equation \eqref{eq:bprimecohomology}, and 
\begin{equation*}
\hmf^*(B\Z/2) \cong \hmf^*[x]
\end{equation*}
is a consequence of Proposition \ref{pro:hfbz2}  and Lemma \ref{lemma:extensionofscalars}.

The fact that $\iota$ is a non-equivariant equivalence gives that, modulo $\rho$, $\iota^*(c) = \tau x$. Without loss of generality, we can assume that $\iota^*(c) = \tau x$ (if not, change $c$ to $c + \rho$).

Now, since $\iota^*$ is a map of algebra and of $\hmf^*$-modules, $\tau^2 x^2 = \iota^*(c^2) = \rho \tau x + \tau \iota^*(b)$. The result follows.
\end{proof}

\begin{thm} \label{thm:comp}
The morphism $\comp : \A_* \rightarrow \Aq_*$ is an $\F$-algebra map. Moreover, it is determined by the formula
\begin{equation*}
\comp(\pxi_i) = \rho^{2^{n-1}-1} \xi_n + \sum_{i=1}^{n} \rho^{2^n - 2^i} \eta_R(\tau)^{2^{i-1}-1} \xi_{n-i}^{2^{i-1}} \tau_{n-1}.
\end{equation*}
\end{thm}

\begin{proof}
The proof relies on the following diagram
\begin{equation} \label{eq:diagramcoactioniota}
\xymatrix@R=2cm@C=2cm{ \frac{\hmf^*[c,b]}{(c^2 + \rho c + \tau b)} \ar[r]^{\iota^*} \ar[d]^{\lambdaq^*} & \hmf^*[x] \ar[d]^{\lambdaq^* = (id \otimes \comp) \lambda^*} \\
\frac{\hmf^*[c,b]}{(c^2 + \rho c + \tau b)} \otimes \Aq_* \ar[r]^{\iota^* \otimes \Aq_*}  & \hmf^*[x] \otimes \Aq_*. }
\end{equation}
which commutes since $\iota^*$ is a map of $\Aq_*$-comodules, and $\lambdaq^* = (id \otimes \comp) \lambda^*$ (see Theorem \ref{thm:twistedext}).

Let's look at the image of $c \in \hmf^*(B'\Z/2)$.
On one hand,
\begin{align*}
(\iota^* \otimes id_{\Aq_*}) \lambdaq c & = \iota^*(c) \otimes 1 + \sum_{i \geq 0} \iota^*(b^{2^i}) \otimes \tau_i  \\
= & \tau x \otimes 1 + \sum_{i \geq 0} \left( \tau^{2^i} x^{2^{i+1}} + \rho^{2^i} x^{2^i} \right) \otimes \tau_i.
\end{align*}
On the other hand,
\begin{align*}
\lambdaq^* \iota^* c = & \lambdaq^*(\tau x)  \\
= & \eta_R(\tau) \lambdaq^*(x) \\
= & \eta_R(\tau) \sum_{i \geq 0} x^{2^i} \otimes \comp(\pxi_i).
\end{align*}
By Diagram \eqref{eq:diagramcoactioniota}, these two quantities have to be equal, giving
\begin{equation*}
\eta_R(\tau) \comp(\pxi_i) = \tau^{2^{i-1}} \tau_{i-1} + \rho^{2^i} \tau_i.
\end{equation*}

The end of the proof is now a computation, which we defer to Lemma \ref{lemma:technicalcomputation}.
\end{proof}

\begin{lemma} \label{lemma:technicalcomputation}
For all $n \geq 0$, the following formula holds in the $\G$-equivariant dual Steenrod algebra
\begin{equation*}
\tau^{2^{n-1}} \tau_{n-1} + \rho^{2^n} \tau_n = \eta_R(\tau) \left(  \rho^{2^{n-1}-1} \xi_n + \sum_{i=1}^{n} \rho^{2^n - 2^i} \eta_R(\tau)^{2^{i-1}-1} \xi_{n-i}^{2^{i-1}} \tau_{n-1} \right).
\end{equation*}
\end{lemma}

\begin{proof}
To show this, we conjugate the entire equation. Denoting $\zeta_i$ for the conjugate element of $\xi_i$ and $\theta_i$ for the conjugate element of $\tau_i$, we get
\begin{align*}
& \chi(\tau^{2^{n-1}} \tau_{n-1} + \rho^{2^n} \tau_n ) \\
= & \chi(\tau^{2^{n-1}}) \theta_{n-1} + \rho^{2^n} \theta_n \\
= & (\rho \theta_0 + \tau)^{2^{n-1}} \theta_{n-1} + \rho^{2^n} \theta_n \\
= & \tau \left(  \rho^{2^{n-1}-1} \zeta_n + \sum_{i=1}^{n} \rho^{2^n - 2^i} \tau^{2^{i-1}-1} \zeta_{n-i}^{2^{i-1}} \theta_{n-1}  \right),
\end{align*}
where the last equality comes from an immediate induction, applying repetitively the relation $\theta_i^2 = \rho \theta_{i+1} + \tau \zeta_{i+1}$ in the dual Steenrod algebra.
\end{proof}

\begin{ex}
The first few formulae read:
\begin{align}
\comp(\pxi_1) = & \tau_0 + \rho \xi_1 \\
\comp(\pxi_2) = & \rho^3 \xi_2 + \rho^2 \xi_1 \tau_1 + \rho \tau_0 \tau_1 + \tau \tau_1 \\
\comp(\pxi_3) = & \rho^7 \xi_3 + \rho^6 \xi_2 \tau_2 + \rho^5( \xi_1 \tau_1 \tau_2 + \xi_1^2 \tau_2 \tau_0) + \rho^4( \tau_2 \tau_1 \tau_0 + \tau \xi_1^2 \tau_2) \\ &+ \rho^3 \tau_2 \tau_1 + \rho^2 \tau^2 \xi_1 \tau_2 + \rho \tau^2 \tau_2 \tau_0 + \tau^3 \tau_2. \nonumber
\end{align}
\end{ex}

\end{appendices}

\bibliographystyle{alpha}
\bibliography{biblio}

\end{document}